\documentclass[12pt,reqno]{amsart}
\usepackage{tikz,dsfont}
\usepackage{amsmath, amssymb, graphicx, mathrsfs, hyperref}
\usepackage{pdfpages}
\usepackage{verbatim} 
\usepackage{enumerate}
\newtheoremstyle{dotless}{}{}{\itshape}{}{\bfseries}{}{ }{} 
\theoremstyle{dotless}

\newtheorem{thm}{Theorem}
\newtheorem{lemma}{Lemma}

\newtheorem{prop}{Proposition}

\newtheorem{corollary}{Corollary}

\begin{document} 

\title{The multiplication table constant and sums of two squares}

\author{Andrew Granville}
\address{D{\'e}partment  de Math{\'e}matiques et Statistique,   Universit{\'e} de Montr{\'e}al, CP 6128 succ Centre-Ville, Montr{\'e}al, QC  H3C 3J7, Canada.}
   \email{andrew.granville@umontreal.ca}  
\author{Cihan Sabuncu}
\address{D{\'e}partment  de Math{\'e}matiques et Statistique,   Universit{\'e} de Montr{\'e}al, CP 6128 succ Centre-Ville, Montr{\'e}al, QC  H3C 3J7, Canada.}
   \email{cihan.sabuncu@umontreal.ca}  
   \author{Alisa Sedunova}
\address{Centre de Recherches Math{\'e}matiques et Statistique,   Universit{\'e} de Montr{\'e}al, CP 6128 succ Centre-Ville, Montr{\'e}al, QC  H3C 3J7, Canada.}
   \email{alisa.sedunova@gmail.com}  
\thanks{The authors are partially supported by a Discovery Grant  from NSERC of Canada. We thank Dimitris Koukoulopoulos and Carl Pomerance for several useful references.}
\dedicatory{Dedicated to Henryk Iwaniec on his 75th birthday}
 
\subjclass[2020]{Primary }

\begin{abstract}  
We will show that the number of integers $\leq x$ that can be written as the square of an integer plus the square of a prime equals $\frac{\pi}{2} \cdot \frac {x}{\log x}$ minus a secondary term of size $x/(\log x)^{ 1+\delta+o(1)}$, where $\delta := 1 - \frac{1+\log\log 2}{\log 2} = 0.0860713320\dots$ is the multiplication table constant. Detailed heuristics  suggest that this secondary term is asymptotic to
 \[
 \frac{1 }{\sqrt{\log\log x}}   \cdot \frac x{(\log x)^{ 1+\delta}} 
 \]
 times a bounded, positive, $1$-periodic, non-constant function of   $\frac{\log\log x}{\log 2}$.
 \end{abstract}

\maketitle

\section{Introduction}
Erd\H os \cite{Er1} showed that there are $o(N^2)$ distinct integers in the $N$-by-$N$ multiplication table with a proof from ``the book''. The best result to date, due to Ford \cite{Fo}, shows that the number of distinct integers is
\[
\asymp \frac{N^2}{(\log N)^{1+\delta}(\log\log N)^{3/2}} ,
\] 
where the \emph{multiplication table constant}  $\delta := 1 - \frac{1+\log\log 2}{\log 2} = 0.0860713320\dots$. The  constant $\delta$ first appears in \cite{Er2}, and has appeared in closely related problems such as in \cite{Ko, Me, Te} as well as some further afield \cite{BP, CP, LP, NP}. In this paper we find that it is important in another seemingly unrelated problem.

Let $N_0(x)$ be the number of integers $n\leq x$ that can be written as 
\[
n=a^2+p^2 \text{ where } p \text{ is a prime and } a \text{ is a positive integer}.
\]

 \begin{thm} \label{thm: newmain}  We have
 \[
  N_0(x) =  \frac{\pi}{2} \cdot \frac {x}{\log x} - \frac {x(\log\log x)^{O(1)}}{(\log x)^{ 1+\delta}}.
 \]
 That is, in estimating  $N_0(x)$  there is a secondary term of size $x/(\log x)^{ 1+\delta+o(1)}$.
 \end{thm}
 
 We will conjecture a formula for the secondary term, using a plausible heuristic, which suggests that it is oscillatory:
\[
N_0(x) = \frac {\pi  }{2 } \cdot \frac {x}{\log x} -
\{ 1+o(1)\}  \frac{\psi_0^*(\frac{\log\log x}{\log 2}) }{\sqrt{\log\log x}}   \cdot \frac x{(\log x)^{ 1+\delta}}   
\]
where $\psi_0^*(t)$ is the 1-periodic continuous, non-constant, positive-valued function,  
\[
 \psi_0^*(t) =   \kappa     \sum_{m\in \mathbb Z}  (2^m\beta)^{ -1-\tau}     (e^{-2^m\beta} - 1 +2^m\beta )  \text{ where } \beta=2^{1-t}
 \]
with $\tau= \frac{\log(\frac 1{\log 2})}{\log 2} =0.52876\dots$ and $\kappa \approx 0.29356\dots$ (defined precisely below).
  
 Let $r_1(n)$ be the number of representations of $n$ as $a^2+p^2$ where $p$ is prime and $a>0$ is an integer, and let 
 \[
 N_r(x):=\# \{ n\leq x: r_1(n)=r\} \text{ for each integer } r\geq 1.
 \]
One can use  the prime number theorem to prove that 
$\sum_{n\leq x} r_1(n) \sim   \frac{\pi}{2} \frac{x}{\log x}$ from which
 one can deduce that 
\[
 N_1(x)\sim   \frac{\pi}{2} \frac {x}{\log x} .
  \]
as $\sum_{r\geq 2} N_r(x)=o(x/\log x)$.
Our main result  gives an estimate for $N_2(x)$, which involves  the multiplication table constant, $\delta$:

 \begin{thm} \label{thm: main} For $\delta := 1 - \frac{1+\log\log 2}{\log 2} = 0.0860713320\dots$ we have
 \[
  N_2(x) = \frac {x(\log\log x)^{O(1)}}{(\log x)^{ 1+\delta}}.
 \]
 Moreover
 \[
 N_1(x) =  \frac{\pi}{2}\cdot \frac {x}{\log x} - \frac {x(\log\log x)^{O(1)}}{(\log x)^{ 1+\delta}}
 \]
 (That is, the difference $N_1(x) - \frac{\pi}{2}\cdot \frac {x}{\log x}$ is negative and of the given size.)
 \end{thm}
 
The second result in Theorem \ref{thm: main} improves the error term in Theorem 2 of \cite{Da}.

Theorem \ref{thm: main}  follows from estimates for the number of pairs of solutions $n=a^2+p^2=b^2+q^2$ with $p,q$ prime, where the   number of distinct prime factors of $n$, denoted by $\omega(n)$, is given:

 \begin{thm} \label{thm: r1(n)choose 2} For any   integer $k\asymp \log\log x$ we have 
 \[
 \sum_{\substack{n\leq x \\ \omega(n)=k}} \binom{r_1(n)}2 \asymp   \frac {x L^{O(1)}}{(\log x)^{3}}\frac{(2 L)^{k}}{k!},
 \]
 where here and throughout it is convenient to define $L:=\log\log x$.
 \end{thm}
 
 This is maximized when $k=2L+O((L\log L)^{1/2})$ whence there are $\frac {x L^{O(1)}}{\log x}$ such pairs, which is consistent with Stephan Daniel's remarkable result \cite{Da}
 \begin{equation} \label{eq: Daniel}
  \sum_{n\leq x} \binom{r_1(n)}2= \frac 98 \cdot \frac {x}{\log x}+O\bigg(\frac {x(\log\log x)^2}{(\log x)^2}  \bigg) .
\end{equation}

 In Corollary \ref{cor: maincontrib} we will show that the main contribution to the sum in \eqref{eq: Daniel}  comes from those integers $n$ with 
 \[
 \omega(n)\sim 2\log\log x \text{ and }  r_1(n)= (\log x)^{\log 4-1+o(1)} . 
 \]

The proofs of the upper bounds in the Theorems (which are given in sections 2  and 4) are based on the methods of \cite{Sa}; here we are less precise but use the techniques in more delicate problems.

We want a heuristic to get a better idea of what is going on.

 \subsection*{Main Conjecture} If $2\leq r=o(\log x)$ we have
 \[
N_r(x) \sim    \frac{\psi_r(\frac{\log\log x}{\log 2}) }{\sqrt{\log\log x}}   \cdot \frac x{(\log x)^{ 1+\delta}} ,
\]
 where $\psi_r(t)$ is a 1-periodic continuous, non-constant, positive-valued function, and for $t\in [0,1)$ we have, for $\beta=2^{1-t}$,
 \[
 \psi_r(t) =  \frac{\kappa}{r!} \sum_{m\in \mathbb Z}   (2^{m}\beta)^{ r-1-\tau}    e^{-2^m\beta} .
 \]
for some constant $\kappa>0$. This function is not monotone decreasing in $r$; indeed if $r\geq 22$ then there are values of $t_+,t_-$ such that $ \psi_r(t^+) > \psi_{r+1}(t^+) $ and $ \psi_r(t^-) < \psi_{r+1}(t^-) $, suggesting that for each $r\geq 22$ there are arbitrarily large values of $x_+,x_-$ such that $ N_r(x^+) > N_{r+1}(x^+) $ and $ N_r(x^-) < N_{r+1}(x^-) $.

As a consequence of our formula for the $\psi_r(t)$, we believe that
 \[
 \# \{ n\leq x: r_1(n)\geq 2\} \sim \frac{\psi_2^*(\frac{\log\log x}{\log 2}) }{\sqrt{\log\log x}}   \cdot   \frac x{(\log x)^{ 1+\delta}} ,
 \]
 where $\psi_2^*(t) $ is the 1-periodic continuous, non-constant, positive-valued function 
 \[
 \psi_2^*(t) = \kappa     \sum_{m\in \mathbb Z}  (2^m\beta)^{ -1-\tau}     ( 1 -(1+2^m\beta)e^{-2^m\beta} ).
 \]
 Another consequence of our formula for the  $\psi_r(t)$, together with  \eqref{eq: sumidff}, suggests that
\[
\frac {\pi  }{2 } \cdot \frac {x}{\log x} -N_1(x) \sim   \frac{\psi_1^*(\frac{\log\log x}{\log 2}) }{\sqrt{\log\log x}}   \cdot \frac x{(\log x)^{ 1+\delta}}   
\]
where $\psi_1^*(t) $ is the 1-periodic continuous, non-constant, positive-valued function 
 \[
 \psi_1^*(t) = \kappa     \sum_{m\in \mathbb Z}  (2^m\beta)^{ -\tau}     ( 1 -e^{-2^m\beta} ).
 \]
We also deduce the conjecture for the secondary term in $N_0(x)$, using \eqref{eq: sumidff}, with  $\psi_0^*(t) = \psi_1^*(t)- \psi_2^*(t)$.
 
 To obtain the correct powers of $\log\log x$ in the Theorems would be considerably more difficult and we believe would require some new ideas.
   
 \subsection{Notation and simple details of sums of squares} Let $r_0(n)$ denote the number of representations of $n$ as the sum of two squares in the upper right quadrant; that is,
 $r_0(n)=\#\{ x\geq 0, y>0: n=x^2+y^2\}$. Corresponding to each such solution we have complex number $x+iy$  and there are  four distinct ``conjugate'' integer solutions, multiplying through by $1,i,-1,-i$ (unless $n=0$). We are considering prime values in sums of two squares, so we wish to define
 \[
 r_0^*(n):=\#\{ a\geq 0,b>0: n=a^2+b^2 \text{ and } (a,b)=1\}
 \]
 which is a multiplicative function with $r_0^{*}(2)=1, r_0^{*}(2^k)=0$ for all $k\geq 2$, and $r_0^{*}(p^k)=1+(\frac{-1}p)$ for all odd primes and $k\geq 1$.
Therefore $r_0^*(n)=0$ if and only if $n\in \mathcal N^c$ where $\mathcal N:=\{ n:\ 4\nmid n, \text{ and } p\nmid n \text{ for all primes } p\equiv 3 \pmod 4\}$. If
$n\in \mathcal N$ then $r_0^*(n)=2^{\omega^*(n)}$
where $\omega^*(n)$ is the number of odd prime factors of $n$. Therefore the values of $ r_0^*(n)$ are either 0 or a power of 2.

If $(m,n)=1$ with  $m=a^2+b^2,\ n =c^2+d^2$ with $a,b,c,d>0$ and $(a,b)=(c,d)=1$, then there correspond two distinct solutions to $mn=x^2+y^2$ with $(x,y)=1$ and $x,y>0$, given by
\[
x+iy=(ac+bd)+i |ad-bc| = |ac-bd|+i (ad+bc)
\]
and as $a^2+b^2$ runs through all the solutions counted by $r_0^*(m)$, and 
$c^2+d^2$ all those solutions counted by $r_0^*(n)$,   then $x^2+y^2$ twice runs  through all the solutions counted by $r_0^*(mn)$ since the pairs $x+iy$ given by $a+ib, c+id$ are the same as those given by $b+ia, d+ic$.

Let $r_1^*(n)$ be the number of representations of $n=a^2+p^2$ where $a>0$ and $(p,a)=1$.

 Let $L:=\log\log x, \lambda=\frac 1{\log 2}$ and $K=\lfloor \lambda L \rfloor=\lfloor \frac{\log\log x}{\log 2} \rfloor$ so that $2^K\in (\frac 12 \log x,\log x]$.
  Denote by $P(n)$ the largest prime factor of an integer $n$ and by $(\frac \cdot \cdot)$ the Legendre symbol. For $k \ge 2$ let $d_k(n) = \# \{ n = x_1 \cdots x_k \}$, $d_2(n) = \tau(n)$.

 \section{Simple applications of known results}
 
 We need to show that most solutions to $n=a^2+p^2$ with $p$ prime have $p\nmid a$:
   
 \begin{lemma} \label{lem: gcd=1} 
 The number of $a^2+p^2\leq x$ with $(a,p)>1$ is 
 \[
 \sum_{n\leq x}( r_1(n)-r_1^*(n)) =O(\sqrt{x}\log\log x).
 \]
 \end{lemma}
 
  \begin{proof}    For each prime $p$ write $a=pm$ so that $m^2\leq x/(p^2+1)$ so the number of such $m$ is $\ll \sqrt{x}/p$, and so in total we have $\leq \sqrt{x} \sum_{p\leq \sqrt{x}} 1/p \ll \sqrt{x}\log\log x$.
  \end{proof}

 \begin{lemma} \label{lem: for smooths}
 Let $z=x^{1/L}$. For any integer $m\geq 1$  we have
 \[
  \sum_{\substack{n\leq x \\ P(n)\leq z \text{ or } P(n)^2|n}} r_0^*(n)^m  \ll_m \frac x{(\log x)^{1000}} 
 \]
 where $P(n)$ is the largest prime factor of $n$.
 \end{lemma}

 \begin{proof}   For $\sigma=1-\frac{\log(L\log L)}{\log z} \in (\frac 12,1)$ we have
\begin{align*}
  \sum_{\substack{n\leq x \\ P(n)\leq z \text{ or } P(n)^2|n }} r_0^*(n)^m &\leq   \sum_{\substack{n\geq 1 \\ P(n)\leq z \text{ or } P(n)^2|n}} r_0^{*}(n)^m (x/n)^\sigma\\
&  = x^\sigma \prod_{p\leq z} \bigg( 1 + \frac{r_0^*(p)^m}{p^\sigma} +\cdots \bigg) \prod_{z<p\leq x} \bigg( 1 + \frac{r_0^*(p^2)^m}{p^{2\sigma}} +\cdots \bigg)\\
& \ll  \frac x{(L\log L)^L} \exp\bigg( \sum_{p\leq z}\frac{2^m}{p^\sigma} \bigg)  \ll \frac x{L^L}
\end{align*}
 as $\sum_{p\leq z}\frac{1}{p^\sigma} \ll  \frac{z^{1-\sigma}}{ (1-\sigma)\log z}\ll L$
 and the result follows.
   \end{proof}

 \section{Counting integers with a given number of prime factors}
 
 Selberg \cite{Se} gave a method for determining the number of integers with a given number of prime factors, and his method can easily be modified to give estimates in many analogous questions, including estimating
 \[
 \pi_\mathcal N(x;k):=\#\{ n\leq x: n\in \mathcal N \text{ and } \omega^*(n)=k\}
 \]
 for all $k\geq 1$. We first need to construct an appropriate family of Dirichlet series:
 \[
\mathcal D_\mathcal N(s,z):= \sum_{\substack{n\leq x\\ n\in \mathcal N}} \frac{(2z)^{\omega^*(n)}}{n^s}  = \bigg( 1 + \frac 1{2^s}\bigg) 
 \prod_{p\equiv 1 \pmod 4} \bigg( 1 + \frac{2z}{p^s-1}\bigg) 
 \]
 since the coefficient of $(2z)^k$  is the sum over $n \in   \mathcal N$ with $ \omega^*(n)=k$, and can be picked out by dividing through by $z^{k+1}$ and integrating around a close contour including $z=0$.  We can compare 
 $\mathcal D_\mathcal N(s,z)$ to $\zeta(s)^z$ (where $\zeta(s)$ is the Riemann zeta-function), the quotient being ``well-behaved'' in a wide enough region of $s$-values, and then the number of such integers up to $x$ can be estimated using Perron's formula and well-known information about $\zeta(s)$. Therefore for any given $A>0$, if $k\leq A\log\log x$ then
\begin{equation} \label{eq: pi_N estimate}
   \pi_\mathcal N(x;k) \sim c_\kappa \cdot \frac x{\log x} \frac{(\frac 12\log\log x)^{k-1}}{(k-1)!},
\end{equation}
 where $\kappa=\frac{k-1}{\log\log x}$ and the constant $c_\kappa>0$ is given by the infinite product
 \[
 c_\kappa:= \frac 3{2^{\kappa+2} \Gamma(\kappa+1)}\cdot   \prod_{p\equiv 1 \pmod 4} \bigg( 1 + \frac{2\kappa}{p-1}\bigg)  \bigg( 1 - \frac 1{p}\bigg)^\kappa
  \prod_{p\equiv 3 \pmod 4}  \bigg( 1 - \frac 1{p}\bigg)^\kappa .
 \]
 We also need an upper bound that is uniform in a wide range.  Here we can imitate the famous argument of Hardy and Ramanujan, so that there exist constants $\gamma_1,\gamma_2>0$ such that for all $k\geq 1$ we have
 \begin{equation} \label{eq: pi_N upperbd}
   \pi_\mathcal N(x;k) \leq  \gamma_1 \frac x{\log x} \frac{(\frac 12\log\log x+\gamma_2)^{k-1}}{(k-1)!}.
\end{equation}
This follows from the prime number theorem for $k=1$ and we then proceed by induction:  Write $n$ counted by 
  $ \pi_\mathcal N(x;k) $ as $2^{e_0}p_1^{e_1}\cdot p_k^{e_k}$, where $e_1,\dots,e_k\geq 1$ and $e_0=0$ or $1$. Then
  $n=p_j^{e_j}m_j$ where $m_j$ is counted by $ \pi_\mathcal N(x/p_j^{e_j};k-1) $, for each $j$, and so
  \[
k \pi_\mathcal N(x;k) \leq \sum_{p^e\leq x/5^{k-1}}  \pi_\mathcal N(x/p_j^{e_j};k-1)
\leq \sum_{\substack{p^e\leq x/5^{k-1} \\ p\equiv 1 \pmod 4}}   \gamma_1 \frac x{p^e\log x/p^e} \frac{(\frac 12\log\log x+\gamma_2)^{k-1}}{(k-1)!}
  \]
  since each prime power $p_j^{e_j}$ is $\geq 5$.
 Then \eqref{eq: pi_N upperbd} follows from the estimate
 \[
 \sum_{\substack{p^e\leq x/5^{k-1} \\ p\equiv 1 \pmod 4 }}\frac {\log x}{p^e\log x/p^e}  \leq \frac 12\log\log x+\gamma_2
 \]
 for some $\gamma_2>0$, which follows from Dirichlet's work on primes in arithmetic progression, and partial summation.
 

 \subsection{Estimating the size of a useful set}
 
 For a given interval $I\subset [0,\frac \pi 2]$ and a given integer $m\geq 1$, define $ \mathcal S_{m,k} (I,R)$ to be the 
 set of $r+is\in \mathbb Z[i]$ with $ \arg(r+is)\in I$, where for $n:=r^2+s^2$ we have
 \[
 n\in [2R^2, 4R^2],   \omega^*(n)=k \text{ and } (n,2m)=1 .
 \]
 If $I\subset (0,\frac \pi 2)$ then $r,s\gg_I R$, which can be useful.

 \begin{prop} \label{prop: keysums} If $m=R^{O(1)}$ and $I\ne \emptyset$ then 
 \[
\#  \mathcal S_{m,k} (I,R) \asymp_{k,I} \frac { R^2}{\log R} \cdot \frac {(\log\log R)^{k-1}}{(k-1)!} \cdot L^{O(1)}.
 \]
 \end{prop}
 
 \begin{proof} 
 It is convenient to write $k=\ell+1$ and then,
 excluding the various $n$ identified in Lemma \ref{lem: for smooths}, we 
  write each such $n=p_1^{e_1}\cdots p_\ell^{e_\ell} q$ where the $p_i$ and $q$ are distinct primes, each $\equiv 1 \pmod 4$
  with $q>R^{1/L}$. Therefore
\[
\# \mathcal S_{m,k} (I,R)\leq \# \mathcal S_{1,k} ([0,\tfrac \pi 2],R) = 
 \sum_{\substack{p_1^{e_1}\leq \dots \leq p_\ell^{e_\ell}\leq q \\ \text{Each } p_i\equiv 1 \pmod 4\\ \frac{2R^2}{p_1^{e_1}\cdots p_\ell^{e_\ell}}<q\leq \frac{4R^2}{p_1^{e_1}\cdots p_\ell^{e_\ell}}}} 2^k
 +O\bigg(  \frac {R^2}{(\log R)^{1000}}  \bigg) .
  \]
  Now the dyadic interval for $q$ must be between $R^{1/L}$ and $R$ and so the number of primes in the interval is 
  \[
  \asymp  \frac{1}{p_1^{e_1}\cdots p_\ell^{e_\ell}} \frac {R^2}{\log R^{1/L}} \leq \frac{L}{p_1^{e_1}\cdots p_\ell^{e_\ell}} \frac {R^2}{\log R}
  \]
  and so 
  \begin{align*}
 \# \mathcal S_{1,k} ([0,\tfrac \pi 2],R) &\ll \frac {2^k L R^2}{\log R} \cdot \sum_{\substack{p_1^{e_1}\leq \dots \leq p_\ell^{e_\ell}\leq R^2\\ \text{Each } p_i\equiv 1 \pmod 4}}  \frac{1}{p_1^{e_1}\cdots p_\ell^{e_\ell}} \\
  &\ll \frac {2^\ell L R^2}{\log R} \cdot \frac 1{\ell!} \bigg( \sum_{\substack{p^e\leq R^2\\ p\equiv 1 \pmod 4}}  \frac{1}{p^e} \bigg)^\ell \\
  & \ll \frac {L R^2}{\log R} \cdot \frac {(\log\log R+O(1))^\ell}{\ell!} = \frac {L^{O(1)} R^2}{\log R} \cdot \frac {(\log\log R)^{k-1}}{(k-1)!}
\end{align*}
for $k\ll \log\log R$.

To get lower bounds we proceed in much the same way.  For a lower bound on $\mathcal S_{1,k} ([0,\tfrac \pi 2],R)$ we need to give a lower bound for the sum over prime powers.  Each term with $p_\ell^{e_\ell}\leq B:=R^{1/L}$ is really in the sum (not just part of an upper bound), and then we have 
  \begin{align*}
\sum_{\substack{p_1^{e_1}\leq \dots \leq p_\ell^{e_\ell}\leq B\\ \text{Each } p_i\equiv 1 \pmod 4}}  \frac{1}{p_1^{e_1}\cdots p_\ell^{e_\ell}} &\geq
 \frac 1{\ell !}\bigg(  \bigg(\sum_{\substack{ p\leq B \\ p\equiv 1 \pmod 4}} \frac 1p \bigg)^\ell - 
  \bigg(\sum_{\substack{ p\leq B \\ p\equiv 1 \pmod 4}}\frac 1{p^2} \bigg) \bigg(\sum_{\substack{ p\leq B \\ p\equiv 1 \pmod 4}} \frac 1p \bigg)^{\ell-2}\bigg) \\
  &= \frac{ (\frac 12\log\log B+O(1))^\ell} {\ell !},
   \end{align*}
   and the same argument as above gives a lower bound which is smaller than the upper bound by a factor of $L^{O_\ell(1)}$.

 For a lower bound on $\# \mathcal S_{m,k} ([0,\tfrac \pi 2],R)$
we need to incorporate that $(n,m)=1$ where $m=R^{O(1)}$ in the argument above, so we need to exclude $O(\log R)$ primes from the sum over the $q$'s, and remove the prime factors of $m$ from the sum over the $p$'s, so that the 
$(\log\log R)^\ell$ is altered to $(\log\log R-\sum_{p|m} \frac 1p)^\ell$. Now $\sum_{p|m} \frac 1p\ll \log\log\log m\ll \log\log\log R$, and so $(\log\log R-\sum_{p|m} \frac 1p)^\ell= (\log\log R)^\ell\cdot L^{O(1)}$.

For a lower bound on $\# \mathcal S_{m,k} (I,R)$ for some non-empty $I\subset [0,\tfrac \pi 2]$ we need to incorporate that $\arg(r+is)\in I$ in the argument above.  We will do so by restricting the angle for $q$. So given an
    $x+iy$ with $x^2+y^2=p_1^{e_1}\cdots p_\ell^{e_\ell}$, we restrict the $u+iv$ for which $u^2+v^2=q$ by ensuring that 
    $\arg(u+iv) \in I-\arg(x+iy)$, an interval of the same length as $I$. Kubilius \cite{Ku} showed that the
    $\{ \arg(u+iv) : u^2+v^2 \text{ is prime}, \in [y,2y]\}$ are  equi-distributed on a dyadic interval as $y\to \infty$, and so we simply need to multiply the above estimates by the factor $|I|/\frac \pi 2$; in particular if $I=[\frac \pi 8, \frac {3\pi} 8]$ then the factor is $\frac 12$.
   \end{proof}
    
   \section{Upper bounds on $\sum_{n\leq x}  \binom{r_1(n)}\ell$}
   
 \begin{lemma} \label{lem: binombounds4R1} For any fixed integer $\ell\geq 1$ and integer $k\ll \log\log x$ we have 
 \[
 \sum_{\substack{n\leq x \\ \omega^*(n)=k}} \binom{r_1(n)}\ell  \ll_\ell   \frac {x L^{O_\ell(1)}}{(\log x)^{\ell+1}}\frac{(2^{\ell-1}L)^{k}}{k!}.
 \]
  \end{lemma}
 
  The case $\ell=2$ of Lemma \ref{lem: binombounds4R1}  gives the upper bound in Theorem \ref{thm: r1(n)choose 2}.
 
\begin{proof} Each $r_1(n)\leq r_0(n)\leq \tau(n)=n^{o(1)}$, so  that
\begin{align*}
 \sum_{\substack{n\leq x \\ \omega^*(n)=k}} \bigg( \binom{r_1(n)}\ell - \binom{r_1^*(n)}\ell  \bigg)
& \ll_\ell \sum_{\substack{n\leq x \\ \omega^*(n)=k}} (r_1(n)-r_1^*(n))r_1(n)^{\ell-1}\\
& \ll x^{o(1)}\sum_{\substack{n\leq x \\ \omega^*(n)=k}} (r_1(n)-r_1^*(n)) \ll  x^{\frac 12+o(1)}
\end{align*}
 by Lemma \ref{lem: gcd=1}. 
 
 For the remaining $n$ we may assume
$P(n)>x^{1/L}$ and $P(n)^2\nmid n$ by Lemma \ref{lem: for smooths}, so we can
  write $n=mp$ where $p=P(n)>x^{1/L}$, so that $P(m)<p$ and $m\leq x^{1-1/L}$.
 If $p=a^2+b^2$ and $m$ has the representations $u_i^2+v_i^2$ for $1\leq i\leq r_0^*(m)$ as the sum of two  positive coprime squares, then every representation of $n$ as the sum of two positive coprime squares is given by 
 \[
 n=mp = (au_i+bv_i)^2 + |bu_i-av_i|^2 
 \]
(note that there exists $j$ such that $u_i=v_j$ and $v_i=u_j$).
 Therefore 
 \[
 r_1^*(n)=\# \{ i: au_i+bv_i \text{ is prime}\}+\# \{ i: |bu_i-av_i| \text{ is prime}\} ,
 \]
 and then letting $\mathcal L_m$ be the set of linear forms $\{ u_ix +v_iy,  v_ix-u_iy: 1\leq i\leq r_0^*(m)\}$, we have
 \begin{align*}
 \sum_{\substack{n\leq x \\ \omega^*(n)=k\\ P(n)>x^{1/L} }} \binom{r_1^*(n)}\ell 
  & = \sum_{\substack{m\leq x^{1-1/L} \\ \omega^*(m)=k-1\\ m=\square+\square}} \sum_{\substack{  x^{1/L} < p\leq x/m \\ p=a^2+b^2 >P(m)}} \binom{r_1^*(mp)}\ell \\
  & \leq  \sum_{\substack{m\leq x^{1-1/L} \\ \omega^*(m)=k-1}} \sum_{\substack{I\subset \mathcal L_m\\ |I|=\ell}} \sum_{\substack{ P(m)< a^2+b^2\leq x/m \\ a^2+b^2 \text{ and } \\ ar+bs \text{ prime for each } rx+sy \in I }} 1 .
  \end{align*}
 Now using the small sieve we have, for $y=x^{1/10L}$ for some small $c>0$,
  \begin{align*}
 \sum_{\substack{ P(m)< a^2+b^2\leq x/m \\ a^2+b^2 \text{ and }  \\ ar+bs \text{ prime for each } rx+sy \in I}} 1 
 &\leq \sum_{\substack{ a,b\leq (x/m)^{1/2} \\ ( (a^2+b^2)\prod_{rx+sy\in I}(ar+bs), \prod_{p\leq y} p)=1  }} 1 \\
 & \ll_\ell  \frac xm \prod_{\ell+2<p\leq y} \bigg( 1 - \frac{ \ell+1+ (\frac{-1}p)} p \bigg)  \prod_{p| R} 
 \bigg( 1 - \frac{1} p \bigg)^{O(\ell)} ,
  \end{align*}
where $R=m\prod_{i<j} (u_iv_j-v_iu_j)(u_iu_j+v_iv_j)\ll x^{O_\ell(1)}$ and so the last Euler product is $L^{O(1)}$. We also have
\[
 \prod_{\ell+2<p\leq y} \bigg( 1 - \frac{ \ell+1+ (\frac{-1}p)} p \bigg) \asymp
  \prod_{p\leq y} \bigg( 1 - \frac{1} p \bigg)^{\ell+1} L(1,(\tfrac{-1}p))^{-1} \asymp \frac 1{(\log y)^{\ell+1}} \asymp \frac {  L^{O_{\ell}(1)}}{(\log x)^{\ell+1}} .
\]
Therefore
\begin{align*}
\sum_{\substack{n\leq x \\ \omega^*(n)=k\\ P(n)>x^{1/L} }} \binom{r_1^*(n)}\ell 
   \ll_\ell  \sum_{\substack{m\leq x^{1-1/L} \\ \omega^*(m)=k-1}} \sum_{\substack{I\subset \mathcal L_m \\ |I|=\ell}}  
   \frac xm \cdot \frac {L^{O_\ell(1)}}{(\log x)^{\ell+1}}  
 \ll_\ell   \frac {x2^{\ell k} L^{O_\ell(1)}}{(\log x)^{\ell+1}}\sum_{\substack{m\leq x^{1-1/L} \\ \omega^*(m)=k-1\\ m=\square+\square}}  \frac 1m   
\end{align*}
since $|\mathcal L_m|= 2^k$. We  split the sum over $m$ into dyadic intervals $(M,2M]$ with $M=2^j, 1\leq j\leq J:=(1-\frac 1L) \frac{\log x}{\log 2}$. Then \eqref{eq: pi_N upperbd} gives
 \[
  \sum_{  j\leq J} \frac 1M \sum_{\substack{M=2^j<m\leq 2M \\ \omega^*(m)=k-1\\   m=\square+\square}}   1
  \ll   \sum_{\substack{  j\leq J \\  M=2^j}}    \frac {1}{M } \cdot \frac{M}{\log M} 
\frac{(\frac 12L+O(1))^{k}}{k!}   \ll  \frac{(\frac 12L)^{k}}{k!} \sum_{j\leq J} \frac 1j \ll \frac{(\frac 12L)^{k+O(1)}}{k!}.
\]
Inserting this above and then combining this with  
Lemma \ref{lem: for smooths}, the result follows.
\end{proof}

 \begin{corollary} \label{cor: UsefulUpperbounds} 
 We have 
\[
  \sum_{\substack{n\leq x \\ \omega^*(n)\geq K }} r_1(n)  \ll  \frac {xL^{O(1)}}{(\log x)^{ 1+\delta}}.
\]
 For any fixed integer $\ell\geq 2$ and integer $\tau, 0\leq \tau \ll \log\log x$ we have 
 \[
 \sum_{\substack{n\leq x \\ \omega^*(n)\leq K-\tau}} \binom{r_1(n)}\ell  \ll_\ell   \frac {x L^{O_\ell(1)}}{(\log x)^{1+\delta}}\cdot \frac {1}{(2^{\ell-1}\log 2)^{\tau}}.
 \]
  \end{corollary}

\begin{proof} By Lemma \ref{lem: binombounds4R1} we have
\[
 \sum_{\substack{n\leq x \\ \omega^*(n)\geq K}} r_1(n) \ll_\ell   \frac {x L^{O_\ell(1)}}{(\log x)^{2}} \sum_{k\geq K} \frac{L^{k}}{k!} \ll  \frac {x L^{O_\ell(1)}}{(\log x)^{2}}   \frac{L^{K}}{K!} = \frac {xL^{O(1)}}{(\log x)^{ 1+\delta}}
 \]
as the summands are decreasing by a factor $\leq \log 2$ for each successive $k\geq K$. Also
\[
\sum_{\substack{n\leq x \\ \omega^*(n)\leq K-\tau}} \binom{r_1(n)}\ell  \ll_\ell   \frac {x L^{O_\ell(1)}}{(\log x)^{\ell+1}}
\sum_{k\leq K-\tau}  \frac{(2^{\ell-1}L)^{k}}{k!}
\ll_\ell   \frac {x L^{O_\ell(1)}}{(\log x)^{\ell+1}} \frac{(2^{\ell-1}L)^{K-\tau}}{(K-\tau)!}
\]
as the summands are decreasing by a factor $\leq 1/\log 4$ as $k$ reduces with $k\leq K$ (starting at $K$, then $K-1$, past $K-\tau$ down to $1$), and
\[
\frac{(2^{\ell-1}L)^{K-\tau}}{(K-\tau)!} \leq \frac{(2^{\ell-1}L)^{K}}{K!} \frac 1{(2^{\ell-1}\log 2)^{\tau}}
\leq \frac { (\log x)^{\ell-1}}{(2^{\ell-1}\log 2)^{\tau}} \cdot \frac{L^{K}}{K!} = \frac { (\log x)^{\ell-\delta}}{(2^{\ell-1}\log 2)^{\tau}} L^{O(1)}
\]
and the result comes from combining the last two displayed equations.
\end{proof}

 \begin{proof}[Proof of upper bounds in Theorem \ref{thm: main}] Taking $\tau=0$ and $\ell=2$ in the second part of 
 Corollary \ref{cor: UsefulUpperbounds} we deduce that 
  \begin{equation} \label{eq:  Bounding}
 \sum_{k\geq K} \sum_{\substack{n\leq x \\ \omega^*(n)=k}} r_1(n) + 
  \sum_{k< K} \sum_{\substack{n\leq x \\ \omega^*(n)=k}} \binom{r_1(n)}2 \ll  \frac {xL^{O(1)}}{(\log x)^{ 1+\delta}}.
 \end{equation}
 Now $1\leq r_1(n)\leq 2\binom{r_1(n)}2$ if $r_1(n)\geq 2$, and so, by \eqref{eq:  Bounding}
 \[
  \#\{ n \le x \colon r_1(n) \ge 2\} = \sum_{r\geq 2} N_r(x) \leq  \sum_{r\geq 2}r N_r(x) \ll \frac {xL^{O(1)}}{(\log x)^{ 1+\delta}}.
 \]
 
 We now show that 
 \[
 \sum_{n\leq x} r_1(n) =\# \{ a^2+p^2\leq x: p \text{ is prime, } a\geq 1\} = \frac {\pi x}{2\log x}+O\bigg(\frac {xL}{(\log x)^2} \bigg)  .
 \]
There are $\ll \frac {x}{(\log x)^2}$ such pairs $(a,p)$
if $a>\sqrt{x}(1-\frac 1{\log x})$. For each integer $a, 1\leq a\leq \sqrt{x}(1-\frac 1{\log x})$, the number 
of primes $\leq (x-a^2)^{1/2}$ is
\[
\pi( (x-a^2)^{1/2}) = \frac{ (x-a^2)^{1/2}}{ \frac 12 \log (x-a^2)} + O\bigg(\frac{ x^{1/2} }{(\log x)^2}\bigg) =
\frac{ 2(x-a^2)^{1/2}}{  \log x} + O\bigg(\frac{ x^{1/2}L }{(\log x)^2}\bigg) 
\]
by the prime number theorem, since $x-a^2\gg x/\log x$. Therefore in total we have
\[
\sum_{a\leq \sqrt{x}(1-\frac 1{\log x}) } \frac{ 2(x-a^2)^{1/2}}{  \log x}  +O\bigg(\frac {xL}{(\log x)^2} \bigg) 
= \frac 2{\log x} \sum_{a\leq \sqrt{x}}  (x-a^2)^{1/2}   +O\bigg(\frac {xL}{(\log x)^2} \bigg)  .
\]
The   sum above equals $\#\{ a,b\geq 1: a^2+b^2\leq x\} =   \frac \pi 4 x+O(\sqrt{x})$, by counting lattice points in the circle, and the claim follows. 

We deduce from this that 
\begin{equation} \label{eq: sumidff}
\frac {\pi  }{2 } \cdot \frac {x}{\log x} -N_1(x)=    \sum_{r\geq 2} r N_r(x)+ O\bigg(\frac {xL}{(\log x)^2} \bigg)   \ll \frac {xL^{O(1)}}{(\log x)^{ 1+\delta}}.
\end{equation}
We also observe that the sum in the middle term is over positive quantities,so the second result of Theorem 1 then follows from the lower bound implicit in the first part of Theorem 1.
  \end{proof}

  \begin{proof}[Completion of the proof of  Theorem \ref{thm: main}, assuming Theorem \ref{thm: r1(n)choose 2}.]
 
 By Corollary \ref{cor: UsefulUpperbounds}, for any integer $\tau, 0\leq \tau \ll \log\log x$ we have 
 \[
 \sum_{\substack{n\leq x \\ \omega^*(n)\leq K-\tau}} \binom{r_1(n)}3  \ll_\ell   \frac {x L^{O_\ell(1)}}{(\log x)^{1+\delta}}\cdot \frac {1}{(2\log 4)^{\tau}},
 \]
 while Theorem \ref{thm: r1(n)choose 2} implies that 
 \[
 \sum_{\substack{n\leq x \\ \omega^*(n)=K-\tau}} \binom{r_1(n)}2 \gg   \frac {x L^{O(1)}}{(\log x)^{3}}\frac{(2 L)^{K-\tau}}{(K-\tau)!}
 \gg \frac {xL^{O(1)}}{(\log x)^{ 1+\delta}(\log 4)^{\tau}}.
\]
 Now if $r\geq 3$ then $ \binom{r}2\leq 3\binom{r}3$ and so
\begin{align*}
N_2(x)  &\geq \#\{n\leq x: r_1(n)=2 \text{ and } \omega^*(n)=K-\tau\}  \\
& \geq  \sum_{\substack{n\leq x \\ \omega^*(n)=K-\tau} } \binom{r_1(n)}2  -3   \sum_{\substack{n\leq x \\ \omega^*(n)\leq  K-\tau  }}  \binom{r_1(n)}3 \\
& \gg \frac {xL^{O(1)}}{(\log x)^{ 1+\delta}(\log 4)^{\tau}} \bigg( 1 - O\bigg( \frac{ L^{O(1)}}{ 2^{\tau}} \bigg)\bigg)\\
 &\gg  \frac {xL^{O(1)}}{(\log x)^{ 1+\delta}(\log 4)^{\tau}} =\frac {xL^{O(1)}}{(\log x)^{ 1+\delta} }
 \end{align*}
 provided $\tau=A \log\log\log x$ with $A$ sufficiently large. 
   \end{proof}
       
  \begin{proof}[Proof of  Theorem \ref{thm: newmain}, assuming Theorem \ref{thm: r1(n)choose 2}]
   Now $N_0(x)=\sum_{r\geq 1} N_r(x)$ and so, by \eqref{eq: sumidff},
 \[
 \frac {\pi  }{2 } \cdot \frac {x}{\log x} -N_0(x)=   \sum_{r\geq 2} (r-1) N_r(x)+ O\bigg(\frac {xL}{(\log x)^2} \bigg)
  \ll \frac {xL^{O(1)}}{(\log x)^{ 1+\delta}}.
 \]
 Moreover, since each $N_r(x)\geq 0$ by definition, this implies that 
  \[
 \frac {\pi  }{2 } \cdot \frac {x}{\log x} -N_0(x)\geq N_2(x) + O\bigg(\frac {xL}{(\log x)^2}\bigg) \geq \frac {xL^{O(1)}}{(\log x)^{ 1+\delta} }
 \]
 by Theorem \ref{thm: main}.
  \end{proof}

  \begin{corollary} \label{cor: maincontrib}
  The main contribution to the sum in \eqref{eq: Daniel}  is given by those integers $n$ with 
  $\omega(n)\sim 2\log\log x$ and $ r_1(n)= (\log x)^{\log 4-1+o(1)} $, that is,
\[
  \sum_{\substack{n\leq x \\ |\omega(n)-2L|\leq C(L\log L)^{1/2}\\ r_1(n)= (\log x)^{\log 4-1+o(1)} }}  \binom{r_1(n)}2\sim \sum_{n\leq x}  \binom{r_1(n)}2.
\]
  \end{corollary}
  
  \begin{proof}
  Theorem \ref{thm: r1(n)choose 2} implies that there exists a constant $C>0$ such that
 \[
 \sum_{\substack{n\leq x \\ |\omega(n)-2L|>C(L\log L)^{1/2} }} \binom{r_1(n)}2     = o\bigg(\frac x{\log x} \bigg).
 \]
  Lemma \ref{lem: binombounds4R1}  with $\ell=3$ implies that 
 \[
 \sum_{\substack{n\leq x \\ |\omega(n)-2L|\leq C(L\log L)^{1/2}\\ r_1(n)> (\log x)^{\log 4-1+\epsilon} }} \binom{r_1(n)}2 \ll \frac 1{(\log x)^{\log 4-1+\epsilon}} \sum_{\substack{n\leq x \\ |\omega(n)-2L|\leq C(L\log L)^{1/2}}} \binom{r_1(n)}3 = o\bigg(\frac x{\log x} \bigg).
  \]
 Finally by \eqref{eq: pi_N estimate} we have
 \begin{align*}
 \sum_{\substack{n\leq x \\ |\omega(n)-2L|\leq C(L\log L)^{1/2}\\ r_1(n)< (\log x)^{\log 4-1-\epsilon} }} \binom{r_1(n)}2 
& \leq  (\log x)^{2\log 4-2-2\epsilon} \sum_{k:\  |k-2L|\leq C(L\log L)^{1/2}}\pi_\mathcal N(x;k) \\
 &\ll  \frac{(\log x)^{2\log 4-2-\epsilon}}{4^{2L}} \sum_{k:\  |k-2L|\leq C(L\log L)^{1/2}}  \frac x{\log x} \frac{(2L)^{2L}}{(2L)!}
 = o\bigg(\frac x{\log x} \bigg).
 \end{align*}
 Inserting these bounds into \eqref{eq: Daniel} gives the result.
  \end{proof}

  It remains for us to prove the lower bounds in  Theorem \ref{thm: r1(n)choose 2}.
 \section{Lower bounds on the sum of $\binom{r_1(n)}2$}
 Notice that
 \[
\sum_{\substack{h\leq x/m \\ p|h \implies p>Y}} \binom{r_1^*(mh)}2 = 
\sum_{\substack{h\leq x/m \\ p|h \implies p>Y }} \sum_{a,b:\  a^2+b^2=h} \sum_{\substack{1\leq i<j\leq r_0^*(m) \\ au_i+bv_i, au_j+bv_j \text{ both prime}}} 1
 \]
 \[
 \geq \sum_{1\leq i<j\leq r_0^*(m)}  \sum_{\substack{a,b\leq \sqrt{x/2m}\\   p|a^2+b^2\implies p>Y \\ au_i+bv_i, au_j+bv_j \text{ both prime}}} 1.
 \]
 Now if $x=Y^u$ with $u\in \mathbb Z$ and $n\leq x$ then $n$ has $\leq [u]$ prime factors $>Y$, and so there are $\leq 2^u$ ways of writing $n=mh$ where $p|h \implies p>Y$. Therefore
 \[
 \sum_{\substack{n\leq x \\ \omega^*(n)=k} } \binom{r_1(n)}2 \geq 2^{-u} \sum_{\substack{m\leq x \\ k-u\leq \omega^*(m)\leq k\\ m=\square+\square} } \sum_{\substack{h\leq x/m \\ p|h \implies p>Y}} \binom{r_1^*(mh)}2
 \]
 \[
\geq  2^{-u} \sum_{\substack{m\leq x \\ k-u\leq \omega^*(m)\leq k} }  \sum_{1\leq i<j\leq r_0^*(m)}  \sum_{\substack{a,b\leq \sqrt{x/2m}\\   (a^2+b^2,Q(Y))=1 \\ au_i+bv_i, au_j+bv_j \text{ both prime}}} 1 ,
 \]
 where $Q(Y):=\prod_{p\leq Y} p$.  We can always write 
\[
u_i+iv_i= (r+is)(g-ih) \text{ and } u_j+iv_j=(r+is)(g+ih) \text{ where } m =(r^2+s^2)(g^2+h^2)
\]
 for some (unique) positive integers $g,h,r,s$, and it is convenient to assume that $(g^2+h^2,r^2+s^2)=1$   so the above is
\[
\geq 2^{-u} \sum_{i+j\in [k-u,k]} \sum_{\substack{g,h \\   \omega^*(g^2+h^2)=i} } \sum_{\substack{r,s\\ m=(g^2+h^2)(r^2+s^2)\leq x \\   \omega^*(r^2+s^2)=j \\
(g^2+h^2,r^2+s^2)=1}}
 \sum_{\substack{a,b\leq \sqrt{x/2m}\\   (a^2+b^2,Q(Y))=1 \\ p,q \text{ both prime}}} 1
\]
\[
\text{ with }  p = a(rg+sh)+b(sg-rh)  \text{ and } q = a(rg-sh)+b(sg+rh).
\]

  We take  
 $R,G$ of the form $R=2^{i/2}x^\eta, G=2^{j/2}x^\eta$ where $1\leq 2^i,2^j\leq x^\eta$ and $\eta>0$ is small,
 to obtain a lower bound given by $\gg (\log x)^2$ sums of the form
\[
 \geq  2^{-u}\sum_{\ell+\ell'\in [k-u,k]}   \sum_{\substack{(g,h)\in \mathcal S_{1,\ell} ([0,\frac \pi 2],G) \\ n:=2gh(g^2+h^2)}}  
  \sum_{\substack{(r,s)\in \mathcal S_{n,\ell'} ([0,\frac \pi 2], R) }}  
  \sum_{\substack { a,b\leq \sqrt{x/2}/4RG \\   (a^2+b^2,Q(Y))=1 \\ p,q\text{ both prime}}} 1,
  \]
  and so, if $u\ll \log\log\log x$ then 
 \begin{equation} \label{eq: setup}
 \sum_{\substack{n\leq x \\ \omega(n)=k} } \binom{r_1(n)}2 \geq (\log x)^2L^{O(1)} \min_{x^\eta<R,G\leq x^{2\eta}} \sum_{\ell+\ell'\in [k-u,k]}   \sum_{\substack{(g,h)\in \mathcal S_{1,\ell} ([0,\frac \pi 2],G) \\ n:=2gh(g^2+h^2)\\ (r,s)\in \mathcal S_{n,\ell'} ([0,\frac \pi 2], R}}  
  \sum_{\substack { a,b\leq \sqrt{x/2}/4RG \\   (a^2+b^2,Q(Y))=1 \\ p,q\text{ both prime}}} 1.
 \end{equation}

 \section{The final sum; conditions on the primes}
 
Now $|r|,|s|\leq 2R$ and $|g|,|h|\leq 2G$  and so, as
\[
u_1=rg+sh, \ v_1=sg-rh, \ u_2=rg-sh, \ v_2=sg+rh, 
\]
we have $|u_i|,|v_i|\leq 8RG$.   

We also have $p = au_1+bv_1  \text{ and } q = au_2+bv_2$ which invert to give
 \[
 \Delta a = pv_2-qv_1,\ \Delta b = qu_1-pu_2  \text{ and } \Delta:=u_1v_2-u_2v_1=2gh(r^2+s^2) .
 \]
 In particular, we have $\Delta^2(a^2+b^2)=F(p,q)$, where 
 \[
 F(x,y):= mx^2-2\xi xy+my^2  \text{ and }  \xi:=u_1u_2+v_1v_2=(r^2+s^2)(g^2-h^2).
 \]
 The final sum above is
 \[
  \sum_{\substack { a,b\leq  \sqrt{x/2}/4RG \\   (a^2+b^2,Q(Y))=1 \\ p,q\text{ both prime}}} 1 =  
  \sum_{\substack { a,b\leq  \sqrt{x/2}/4RG  \\ p,q\text{ both prime} \\   ( \tfrac{F(p,q)}{\Delta^2} ,Q(Y))=1 }} 1,
 \]
 so we need to gain some understanding of what integers divide $F(p,q)$.

\subsection{Congruence  conditions on $p$ and $q\pmod \Delta$}
To guarantee that $a$ and $b$ are integers, we give the following lemma  without proof:

 \begin{lemma} \label{lem: ell} With the notation as above, 
$  a = \frac{pv_2-qv_1}{\Delta}$ and $b = \frac{qu_1-pu_2}{\Delta}$  are integers if and only if  $q\equiv \ell p  \pmod \Delta$ where, defining
 $i$ by $s\equiv ir \pmod{(r^2+s^2)^2}$,
 \begin{align*}
 (gi-h) \ell &\equiv (h+ig) \pmod{r^2+s^2} \\
 \  \ell &\equiv 1 \pmod {h(2,h)},\\ 
 \text{ and   }\qquad  \ell &\equiv -1 \pmod {g(2,g)}. 
\end{align*}
We call this value $\ell^{(g,h,r,s)} \pmod \Delta$.
Note that $\Delta = (r^2+s^2) \cdot h(2,h) \cdot g(2,g)$ and these three moduli are pairwise coprime.
 \end{lemma}

 We have $F(p,q)\equiv F(1,\ell)p^2 \pmod {\Delta^2}$.
Squarefree $d$ divides $a^2+b^2$ if and only if  
$q\equiv \ell_d p \pmod {d \Delta}$ where $d \Delta^2$ divides $F(1,\ell_d)$.
Moreover $\ell_d\equiv \ell \pmod \Delta$.

 \begin{lemma} \label{lem: sortv(.)}With the notation as above, for squarefree $d$ define $\nu(d)=\# \mathcal N(d)$ where 
\[
\mathcal N(d):= \{ \lambda \pmod {d\Delta}: \lambda\equiv \ell \pmod \Delta, (\lambda,d)=1 \text{ and } d| \tfrac{F(1,\lambda)}{\Delta^2} \} .
\]  
Then $\nu(\cdot)$ is a multiplicative function on the squarefree $d$, with 
 \[
\nu(p):= \begin{cases}
 1+(\frac {-4}p) &\text{ if } p\nmid m;\\
1 &\text{ if } p|r^2+s^2\\
0 &\text{ if } p|g^2+h^2
\end{cases}
\]
 \end{lemma}
 
 \begin{proof} Now $p| \tfrac{F(1,\lambda)}{\Delta^2}$ if and only if
 $p\Delta^2|F(1,\lambda)=m\lambda^2-2\xi \lambda+m$. This is a quadratic polynomial of discriminant 
 $4\xi^2-4m^2=-4\Delta^2$. Therefore:
 
-- If prime $p$ does not divide $2\Delta m$ then $\nu(p)=1+(\frac {-1}p)$. 
 
-- If prime $p$   divides $g^2+h^2$ (which divides $m$ but not $\xi\Delta$) then $\lambda\equiv 0 \pmod p$.
 
 \noindent Now suppose prime $p$   divides  $\Delta$ then $p$ divides exactly one of $H=r^2+s^2, g(2,g)$ or $h(2,h)$.
 Then we can write $\lambda=\ell+k H$ for some integer $k$, and so
 \[
 m\lambda^2-2\xi \lambda+m\equiv (m\ell^2-2\xi\ell +m) +(m \ell -\xi)  2kH +mk^2H^2 \pmod{pH^2}.
 \]
 
-- If $H=r^2+s^2$ then the congruence becomes  $k\equiv 0 \pmod p$
 
 -- If $H=(2,g)g$ write $\ell=-1$ and  the congruence becomes 
 $(hk/2)^2+1\equiv 0 \pmod p$ if $p\ne 2$, and $k\equiv 1 \pmod 2$ if $p=2$.

 -- If $H=(2,h)h$ write $\ell=1$ and  the congruence becomes 
$(gk/2)^2+1\equiv 0 \pmod p$
 if $p\ne 2$, and $k\equiv 1 \pmod 2$ if $p=2$.

Therefore $\nu(p)=1+(\frac {-4}p)$ if $p|2gh$.
\end{proof}

\subsection{The size of $p$ and $q$}
Let $X= \sqrt{x}/32$.
Suppose that $g+ih=\sqrt{g^2+h^2}\cdot e(\theta)$ with $\theta\in [\frac \pi 8, \frac {3\pi} 8]$ and so
$\sin(2\theta)\geq \frac 1{\sqrt{2}}$ which gives 
$2gh=  (g^2+h^2)\sin(2\theta)\geq 2G^2\sin(2\theta)\geq \sqrt{2}G^2$. Now $r^2+s^2\geq 2R^2$
and so $\Delta\geq 2 \sqrt{2}G^2R^2$ which gives
\[
2 \sqrt{2}G^2R^2 \max\{ |a|,|b|\} \leq \Delta \max\{ |a|,|b|\} \leq 16RGX=RG\sqrt{x}/2,
\] 
that is, $|a|,|b|\leq   \sqrt{x/2}/4RG$, as desired.
 Therefore we have
\[
\sum_{\substack { a,b\leq  \sqrt{x/2}/4RG  \\ p,q\text{ both prime} \\   ( \tfrac{F(p,q)}{\Delta^2} ,Q(Y))=1 }} 1
\geq \sum_{\substack {\text{Primes }  p,q\leq X  \\  q\equiv \ell^{(g,h,r,s)} p  \pmod \Delta \\  ( \tfrac{F(p,q)}{\Delta^2} ,Q(Y))=1 }} 1
 \]
 
 For a given $\ell$-value define
 $\mathcal P_{\ell, \Delta}(X):= \{ \text{Primes } p,q\leq X: q\equiv \ell p  \pmod \Delta\}$, and 
  \[
 \mathcal A=\mathcal A_{(g,h,r,s)}:=\{ \tfrac{F(p,q)}{\Delta^2}:\  p,q\in \mathcal P_{\ell, \Delta}(X)\} ,
 \]
 where $\ell= \ell^{(g,h,r,s)}$,
 counting with multiplicity. We define $\mathcal A_d=\{  a\in \mathcal A: d|a\}$ for squarefree $d$, so that
 \[
\mathcal A_d= \bigcup_{\lambda\in \mathcal N(d)} \{ \tfrac{F(p,q)}{\Delta^2}:\  p,q\leq X,\ q\equiv \lambda p  \pmod {d\Delta}\},
\]
and then
\[
\sum_{\substack {\text{Primes }  p,q\leq X  \\  q\equiv \ell^{(g,h,r,s)} p  \pmod \Delta \\  ( \tfrac{F(p,q)}{\Delta^2} ,Q(Y))=1 }} 1 
=\mathcal S(\mathcal A_{(g,h,r,s)}, Y ) 
\]
with the usual sieve theory notation.

 \section{Sieve theory} 

To use sieve theory we will need to estimate the size of the sets $\# \mathcal A$ and $\# \mathcal A_d$:
\begin{align*}
\# \mathcal A &=\sum_{p,q\leq X} \frac 1{\phi(\Delta)} \sum_{\chi \pmod \Delta} \chi(\ell p\overline{q})
= \frac 1{\phi(\Delta)} \sum_{\chi \pmod \Delta} \chi(\ell)\sum_{p\leq X} \chi(p) \cdot \sum_{q\leq X} \chi(\overline{q}) \\ 
& = \frac 1{\phi(\Delta)} \sum_{\chi \pmod \Delta} \chi(\ell)|\pi(X,\chi)|^2
= \frac 1{\phi(\Delta)}    \sum_{k| \Delta} \sum_{\substack{\chi \pmod k \\ \chi \text{ primitive}}} \chi(\ell)|\pi(X,\chi)|^2+O(X),
\end{align*}
since  if $\chi'$ is induced by the primitive characters $\chi$ then
$|\pi(X,\chi)-\pi(X,\chi')|\leq \sum_{p|\Delta} 1\ll \log X$ and  $|\pi(X,\chi)|^2 -|\pi(X,\chi')|^2\ll X$;
and similarly
\begin{align*}
\# \mathcal A_d&=\sum_{\ell\in \mathcal N(d)} \bigg( \frac {1}{\phi(d\Delta)}    \sum_{k| d\Delta} \sum_{\substack{\chi \pmod k \\ \chi \text{ primitive}}} \chi(\ell)|\pi(X,\chi)|^2+O(X) \bigg)\\
&=\frac {\nu(d)\phi(\Delta)}{\phi(d\Delta)}  \# \mathcal A + O(\text{Error}(d))
\end{align*} 
where
\[
\text{Error}(d):=  \frac {\nu(d)}{\phi(d\Delta)}  \sum_{\substack{k| d\Delta\\ k\nmid \Delta}} \sum_{\substack{\chi \pmod k \\ \chi \text{ primitive}}} |\pi(X,\chi)|^2+O(\nu(d) X).
\]
The \emph{fundamental lemma of sieve theory} then gives that     
\[
\mathcal S(\mathcal A, Y ) = \prod_{p\leq Y} \bigg( 1 - \frac {\nu(p)\phi(\Delta)}{\phi(p\Delta)} \bigg) \# \mathcal A \{ 1+O(v^{-v})\}+
O\bigg( \sum_{1<d\leq Y^v} |\text{Error}(d)| \bigg) .
\]
We also write $ \# \mathcal A = \pi(X)^2/\phi(\Delta) +  \text{Error(1)}$ and so
\[
\mathcal S(\mathcal A, Y ) = \prod_{p\leq Y} \bigg( 1 - \frac {\nu(p)\phi(\Delta)}{\phi(p\Delta)} \bigg) \frac{ \pi(X)^2}{\phi(\Delta)} \{ 1+O(v^{-v})\} + 
 \text{Total Error}
\]
 since $\nu(d)\leq r_0(d)$,
\begin{align*}
\text{Total Error} & \ll    \sum_{d\leq D:=Y^v}\bigg(  \frac {\nu(d)}{\phi(d\Delta)}  \sum_{\substack{k| d\Delta\\ k\nmid \Delta \text{ if } d>1 \\ k>1  \text{ if } d=1}} \sum_{\substack{\chi \pmod k \\ \chi \text{ primitive}}} |\pi(X,\chi)|^2+  r_0(d) X \bigg) \\ 
 & \ll  \sum_{\substack{1<k\leq D\Delta}} \sum_{\substack{\chi \pmod k \\ \chi \text{ primitive}}} |\pi(X,\chi)|^2
\bigg( \frac{1_{k|\Delta}}{ \phi(\Delta)}  + 1_{k\nmid \Delta} \sum_{\substack{ 1<d\leq D\\ k| d\Delta  }}  \frac {\nu(d)}{\phi(d\Delta)}  \bigg)+    XD .
\end{align*} 
Now $k| d\Delta$ means $d=e \frac k{(k,\Delta)}$ for some $e\geq 1$, and so the sum is
\[
\leq \sum_{e\leq D}  \frac {\nu( \frac k{(k,\Delta)}\cdot e)}{\phi([\Delta,k]\cdot e)} \leq \frac {\nu( \frac k{(k,\Delta)} )}{\phi([\Delta,k] )}   \sum_{e\leq D} \frac {\nu(e)}{\phi(e)} \ll \frac {\nu( \frac k{(k,\Delta)} )}{\phi([\Delta,k] )} \log Y.
\]
and so, since $\nu(.)\leq r_0(.)$,
\[
\text{Total Error}\ll  
(\log Y)\cdot  \sum_{\substack{1<k\leq D\Delta}}  \frac {r_0( \frac k{(k,\Delta)} )}{\phi([\Delta,k] )}  \sum_{\substack{\chi \pmod k \\ \chi \text{ primitive}}} |\pi(X,\chi)|^2 +    XD .
\]

Since $\Delta$ is even, $\phi(\Delta)\phi(p)\leq \phi(p\Delta)$ and $\nu(.)\leq r_0(.)$
 by Lemma \ref{lem: sortv(.)}, we have
\begin{align*}
\frac{1}{\phi(\Delta)}\prod_{p\leq Y} \bigg( 1 - \frac {\nu(p)\phi(\Delta)}{\phi(p\Delta)} \bigg) &\geq \frac{1}{2 \Delta} \prod_{2<p\leq Y} \bigg( 1 - \frac{\nu(p)}{\phi(p)} \bigg) \\
&\geq \frac{1}{2 \Delta} \prod_{2<p\leq Y} \bigg(1- \frac{1+(\frac{-1}{p})}{p-1}\bigg) \gg \frac{1}{\Delta \log Y}
\end{align*}
by Mertens' Theorem.
 We will assume that $v$ is large enough that $1+O(v^{-v})$ is close to $1$, and write $y=Y^v$, so that the results in this section
together with   the prime number theorem, imply that
\begin{equation} \label{eq: sieve}
 \sum_{\substack { a,b\leq \sqrt{x}/6RG \\   (a^2+b^2,Q(Y))=1 \\ p,q\text{ both prime}}} 1\gg 
   \frac{X^2  L^{O(1)}}{\Delta (\log X)^3}  + 
 O\bigg( (\log Y)\cdot  \sum_{\substack{1<k\leq D\Delta}}  \frac {r_0( \frac k{(k,\Delta)} )}{\phi([\Delta,k] )}  \sum_{\substack{\chi \pmod k \\ \chi \text{ primitive}}} |\pi(X,\chi)|^2 +    XD\bigg)
 \end{equation}
as $ Y = x^{1/u} \ge x^{1/O(\log \log \log x)} $.
  
\subsection{The sum  of the error terms} 
When we insert this into \eqref{eq: setup} we will sum this formula over many quadruples $g,h,r,s$ for which 
$g^2+h^2, r^2+s^2\ll x^{2\eta}$ for a given $X$-value. The error term in this sum only depends of $\Delta$, and the number of choices of
$g,h,r,s$ for which  $\Delta = (r^2+s^2)2gh$, when $\Delta$ is odd and squarefree, equals 

\[
	\sum_{2gha = \Delta} r_0(a) \le \tau(\Delta) d_3 (\Delta)
\]
and so the error term for that sum is 
\begin{align*}
	&\ll (\log Y) \sum_{2 \mid \Delta \ll x^{4 \eta}} \mu^2(\Delta) \tau(\Delta) d_3(\Delta) \sum_{1 < k \le D \Delta} \frac{r_0(\frac{k}{(k,\Delta)})}{\phi([\Delta,k])} \sum_{\substack{\chi \pmod k \\ \chi \text{ primitive}}} |\pi(X,\chi)|^2 + XD \sum_{2 \mid \Delta \ll x^{4 \eta}} \tau(\Delta) d_3(\Delta)\\
	&\ll (\log Y) \sum_{1 < k \le Dx^{4 \eta}} \sum_{\substack{\chi \pmod k \\ \chi \text{ primitive}}} |\pi(X,\chi)|^2 \sum_{2 \mid \Delta \ll x^{4 \eta}} \mu^2(\Delta) d_3(\Delta) \frac{\tau([\Delta,k])}{\phi([\Delta,k])} + X D x^{4 \eta} (\log x)^5
\end{align*}
and then
\begin{align*}
	&\sum_{\Delta \ll x^{4 \eta}} \mu^2(\Delta) d_3(\Delta) \frac{\tau([\Delta,k])}{\phi([\Delta,k])}  = \frac{\tau(k)}{\phi(k)} \sum_{\Delta \ll x^{4 \eta}} \mu^2(\Delta) d_3(\Delta)  \frac{\tau(\frac{\Delta}{[\Delta,k]})}{\phi(\frac{\Delta}{[\Delta,k]})} \\
	&\ll \frac{\tau(k)}{\phi(k)} \prod_{p \mid k} (1+d_3(p)) \prod_{\substack{ p \nmid k \\ p \ll x^{4 \eta}}} \biggl( 1+ \frac{6}{p}\biggr) \ll \frac{\tau(k) d_4(k)}{\phi(k)} (\log x)^6.
\end{align*}
Therefore the sum of the error terms is

\[
	\ll (\log x)^6 \sum_{1 < k \le X^{1/2-\epsilon}} \frac{\tau(k) d_4(k)}{\phi(k)}  \sum_{\substack{\chi \pmod k \\ \chi \text{ primitive}}} |\pi(X,\chi)|^2.
\]

Now
 \[
\frac {1}{\phi(k)}   \sum_{\substack{\chi \pmod k \\ \chi\ne\chi_0}}  |\pi(X,\chi)|^2 =  \sum_{(a,k)=1} \bigg| \pi(X;k,a) -\frac{\pi(X)}{\phi(k)}\bigg|^2 .
 \]
Since $\pi(X;m,a)\ll \frac{\pi(X)}{\phi(m)}$ by the Brun-Titchmarsh theorem if $m<X^{1-c}$ the last displayed line is 
\[
  \ll    \frac X{\log X} \max_{(a,k)=1}   \bigg| \pi(X;k,a) -\frac{\pi(X)}{\phi(k)}\bigg| ,
\]
and so  the sum of the error terms is
\[
\ll X(\log X)^{5}  \sum_{ 1<k\leq   X^{1/2-\epsilon}}   \tau(k)d_4(k)   \max_{(a,k)=1}   \bigg| \pi(X;k,a) -\frac{\pi(X)}{\phi(k)}\bigg| .
 \]
The square of the sum  is, by Cauchying, 
\[
 \leq  \biggl(\sum_{ 1<k\leq   X^{1/2-\epsilon}}  \frac{ \tau(k)^2d_4(k)^2}{\phi(k)} \biggr) \biggl(  \sum_{ 1<k\leq   X^{1/2-\epsilon}}\phi(k)\max_{(a,k)=1}   \bigg| \pi(X;k,a) -\frac{\pi(X)}{\phi(k)}\bigg|^2 \biggr)
 \]  
 \[
 \ll  (\log X)^{64}\cdot \frac X{\log X}  \sum_{ 1<k\leq   X^{1/2-\epsilon}} \max_{(a,k)=1}   \bigg| \pi(X;k,a) -\frac{\pi(X)}{\phi(k)}\bigg| \ll \frac{X^2}{(\log X)^{2A+10}}
 \]  
  and so  the sum of the error terms is
\[
\ll \frac{X^2}{(\log X)^{A}} + XD  x^{4 \eta} (\log x)^{ 5 }\ll \frac{X^2}{(\log X)^{A}}.
 \]

\subsection{The main terms} Since we are aiming for a lower bound we can select a convenient subset of all the possibilities for $g,h,r,s$.
We take  $X= \sqrt{x}/32$ and  $\arg(g+ih)\in [\frac \pi 8, \frac {3\pi} 8]$ so that $gh\asymp G^2$ and $\Delta\asymp G^2R^2$.
 Therefore, as $\log X\asymp \log x$, we have
 \[
 \sum_{\substack{n\leq x \\ \omega(n)=k} } \binom{r_1(n)}2 \gg 2^{-u}  \frac{x}{  (\log x)^3} \sum_{R,G}   \sum_{\ell+\ell'\in [k-u,k]}  \frac{ 1}{G^2}   \sum_{\substack{(g,h)\in \mathcal S_{1,\ell} ([\frac \pi 8, \frac {3\pi} 8],G) \\ n:=2gh(g^2+h^2)}}  
\frac{ 1}{R^2} \# \mathcal S_{n,\ell'} ([0,\frac \pi 2],R) 
 \]
 \[ 
 +O\bigg(\frac{x}{(\log x)^{A}} \bigg) ,
  \]
  and the main term here is, as $k\ll \log\log x$ and $u\ll \log\log\log x$,
  \begin{align*}
  &\gg L^{O(1)}  \frac{x}{  (\log x)^3} \sum_{R,G}   \sum_{\ell+\ell'\in [k-u,k]}  \frac{ 1}{G^2} \# \mathcal S_{1,\ell} ([\frac \pi 8, \frac {3\pi} 8],G) 
\frac {1}{\log R} \cdot \frac {(\log\log R)^{\ell'-1}}{(\ell'-1)!} \cdot L^{O(1)} \\
&\gg    \frac{x}{  (\log x)^5} \sum_{R,G}   \sum_{\ell+\ell'\in [k-u,k]} 
\frac {(\log\log x)^{\ell-1}}{(\ell-1)!} \cdot 
  \frac {(\log\log x)^{\ell'-1}}{(\ell'-1)!}   \cdot L^{O(1)}  
\end{align*}
  by two applications of Proposition \ref{prop: keysums} and as $\log R,\ \log G \asymp \log x$,
\[
\gg    \frac{x}{  (\log x)^3}    \sum_{n\in [k-u-2,k-2]} 
\frac {(2\log\log x)^n}{n!}    \cdot L^{O(1)}  =  \frac{x}{  (\log x)^3}  
\frac {(2\log\log x)^k}{k!}    \cdot L^{O(1)}
\]
as there are $\asymp \log x$ intervals for both $G$ and $R$, and using the binomial theorem, and then assuming that 
$k\asymp \log\log x$, which  completes the proof of Theorem   \ref{thm: r1(n)choose 2}. \qed

 \section{Basic heuristic for the conjecture}
   The basic heuristic claims that ``typically'', 1 in $\log \sqrt{x}$, of the representations of $n=a^2+b^2$ with $(a,b)=1$ have $b$ a prime.

 Therefore if $r_0^*(n)=R$ then
 \begin{align*}
  \text{Prob} (r_1^*(n)=r) &= \binom{R}r \bigg( \frac 1{\log \sqrt{x}} \bigg)^r \bigg( 1- \frac 1{\log \sqrt{x}} \bigg)^{R-r} \\
  & =
 \frac{R^r}{r!} \prod_{j=1}^{r-1} \bigg( 1 - \frac jR\bigg) \cdot \bigg( \frac 2{\log x-2} \bigg)^r \cdot \exp\bigg( R \log \bigg(1-\frac 2{\log x}\bigg)  \bigg)\\
  & =
 \frac{1}{r!}    \bigg( \frac {2R}{\log x} \bigg)^r \cdot \exp\bigg( -\frac {2R}{\log x} +O  \bigg(\frac {R}{(\log x)^2} +\frac{r^2}R+\frac r{\log x}\bigg)  \bigg) .
 \end{align*}
 There is a remarkable change in behaviour here when $R$ is close to $\frac 12 \log x$; indeed if $R=\alpha\cdot \frac 12 \log x$ and $r$ is fixed then the above
 \[
 e^{-\alpha} \frac{\alpha^r}{r!}   \bigg( 1 +O  \bigg( \frac 1{\log x}\bigg)  \bigg) .
 \]
  
Using \eqref{eq: pi_N estimate} we expect that 
\begin{align*} 
\pi_\mathcal N(x;k,r):&= \# \{ n\leq x: \omega^*(n)=k \text{ and } r_1(n)=r\} \\
&\asymp    \frac x{\log x} \frac{ (\tfrac 12\log\log x)^{k-1}} {(k-1)!}  \cdot \frac{2^{(k+1)r}}{r!}   \frac 1{(\log x)^r}  \exp\bigg(  - \frac {2^{k+1}}{\log x} \bigg) \\
& = \frac{2^{2r}}{r!}    \frac x{(\log x)^{r+1}} \cdot T_k \text{ where } T_k:= \frac{ (2^{r-1} \log\log x)^{k-1}} {(k-1)!}  \exp\bigg(  - \frac {2^{k+1}}{\log x} \bigg) 
 \end{align*} 
and we now  wish to sum over $k$.   If $r\geq 2$ then the $ \frac{ (2^{r-1} \log\log x)^{k-1}} {(k-1)!}$ is increasing whenever   $2^k\leq \log x$, by a factor $>2\log 2>1$ at each consecutive $k$; and the $\exp(- \frac {2^{k+1}}{\log x})$ term can be ignored in this range (as it is $\asymp 1$), but for larger $k$ it decreases rapidly and so the maximum here occurs when $2^k\asymp \log x$; that is, $k=K+m$ with $m=O(1)$ (and so we can take $\kappa=\lambda$ in  \eqref{eq: pi_N estimate}). Thus  $\alpha_k =2^m\alpha_K$ (where $\alpha_k=2^{k+1}/\log x$) and so
\[
\pi_\mathcal N(x;k)\sim c_\lambda \frac x{\log x} \frac{ (L/2)^{k-1}} {(k-1)!} \sim 
 \kappa \frac x{(\log x)^{1+\delta}}    \frac { \alpha_k^{-1-\tau}}{\sqrt{ \log\log x }  } ,
\]
where   $\kappa=c_\lambda \sqrt{ (2\lambda)^3/ \pi} $ (and  $\kappa \approx 0.29356\dots$, which was obtained by taking all primes $p<10^7$ in the Euler product)
since, by Stirling's formula
\[
\frac{ (L/2)^k} {k!} \sim \frac {(eL/2k)^k}{\sqrt{2\pi k}} \frac{2^{\delta L/\log 2}} {(\log x)^{\delta} }\sim
\frac {(e\cdot 2^\delta)^{\lambda L  -k}}{\sqrt{2\pi k} (\log x)^{\delta} } =
\frac {2\lambda\cdot \alpha_k^{-1-\tau}}{\sqrt{2\pi k} (\log x)^{\delta} }.
\]
Therefore we expect that
 \[
\pi_\mathcal N(x;K+m,r)\sim     \frac {\kappa   x}{(\log x)^{1+\delta}  \sqrt{\log\log x}}   \cdot \frac{e^{- 2^m\alpha_K} (2^m\alpha_K)^{ r-1-\tau}}{r!}
 \]
 and so we guess that 
 \[
 N_r(x)=\sum_k \pi_\mathcal N(x;k,r) \sim    \psi_r(\tfrac{\log\log x}{\log 2}) \cdot  \frac { x}{(\log x)^{1+\delta}  \sqrt{\log\log x}} ,
 \]
 where $\psi_r(t)$ is a continuous (bounded) function of period $1$, and for $t\in [0,1)$ we have, for $\beta=2^{1-t}$,
 \[
 \psi_r(t) =   \frac{\kappa }{r!} \sum_{m\in \mathbb Z}   (2^{m}\beta)^{ r-1-\tau}    e^{-2^m\beta} .
 \]
 It is evident that this is invariant under the map $\beta\to 2\beta$ which is equivalent to  $t\to t-1$; that is, this definition is
 $1$-periodic.

\subsection*{$N_r(x)$ vs $N_{r+1}(x)$} Is it true that $N_r(x)>N_{r+1}(x)$ if $x$ is large? With our heuristic above we expect that 
$N_r(x)>N_{r+1}(x)$ for all large $x$ if and only if $\psi_r(t)>\psi_{r+1}(t)$ for all $t\in [0,1)$ if and only if
 \begin{equation} \label{eq: selfsimilar}
 f_{ r-1-\tau}  (\beta)   \geq    \frac{1}{r+1} f_{ r-\tau}  (\beta)   
  \text{ where } f_R(\beta):=\sum_{m\in \mathbb Z}   (2^{m}\beta)^{R}    e^{-2^m\beta}  
 \end{equation}
for all $\beta\in (1,2]$.  The sum $f_R(\beta)$ is invariant under multiplication by $2$ so we extend the definition all $\beta>0$.The largest term in this sum comes when $2^m\beta\in [R\log 2, R\log 4]$ so, wlog
$\beta=\lambda R, \lambda \in (\log 2, \log 4]$.  Therefore the sum becomes 
\[
f_R(\beta):=\sum_{m\in \mathbb Z}   (2^{m}\lambda R)^{R}    e^{-2^m\lambda R} 
= \beta^{R}    e^{-\beta} \bigg( 1+ O \bigg(   e^{-(\log 2-\frac 12\lambda)R}   +e^{(\log 2- \lambda)R}    \bigg) ,
\]
and one can use this to show that 
\[
f_R(R\log 2)\sim 2^{1-\delta R} \cdot R^{R}    e^{-R}  \leq f_R(\beta)\leq f_R(R) \sim R^{R}    e^{-R}
\]
 (and so $f_{ r-1-\tau}  (\beta)/r! \lesssim 1/(2\pi)^{1/2} r^{\frac 32+\tau}$). Substituting  this estimate into 
 \eqref{eq: selfsimilar} with $\beta=\lambda(r-\tau)$ and we more-or-less obtain
 \[
  \lambda(r-\tau)\leq   r+1; \text{ that is } \lambda \leq 1+O(1/r).
 \]
 Therefore if say $\beta=\frac 76(r-\tau)$ then we expect that if $r$ is sufficiently large then we can determine values of $x$ for which  $N_r(x)<N_{r+1}(x)$ if $x$ is large. Indeed calculations reveal that if $r\geq 22$ then
 $f_{ r-1-\tau}  (\frac 76(r-\tau))   \geq    \frac{1}{r+1} f_{ r-\tau}  (\frac 76(r-\tau))$; that is, if 
for a sufficiently large integer $m$, 
 \[
 x=\exp\bigg( \frac {6\cdot 2^m}{7   (r-\tau)} \bigg)  \text{ then } N_r(x)<N_{r+1}(x) ,
 \]
 where for all $r\geq 2$ if 
  \[
 x=\exp\bigg( \frac {6\cdot 2^m}{5   (r-\tau)} \bigg)  \text{ then } N_r(x)>N_{r+1}(x).
 \]
Further calculations suggest that if $2\leq r\leq 21$ then $N_r(x)>N_{r+1}(x)$ for all sufficiently large $x$ (but not for $r=22$).

Now
\[
\int_{\beta=1}^2 f_R(\beta)\frac{d\beta}{\beta}
=\sum_{m\in \mathbb Z}  \int_{\beta=1}^2 (2^{m}\beta)^{R}    e^{-2^m\beta} \frac{d(2^{m}\beta)}{2^{m}\beta}
=\sum_{m\in \mathbb Z}  \int_{t=2^m}^{2^{m+1}} t^{R-1}    e^{-t}  dt = \int_{t=0}^\infty t^{R-1}    e^{-t}  dt =\Gamma(R),
\]
so that $\frac 1{r!}f_{r-1-\tau}(\beta)$ is, on average, $\frac{\Gamma(r-1-\tau)}{r!\log 2}\sim \frac 1{r^{2+\tau}\log 2}$, and we expect that for large $X$,
 \[
  \int_{x=X}^{X^2}  \frac{ N_r(x)   } { x/(\log x)^{\delta}} \frac{dx}x \sim   \frac{\kappa\, \Gamma( r-1-\tau)/r!}  {\sqrt{\log\log X} } .
 \]

     \bibliographystyle{amsplain}

 \end{document}